\documentclass[12pt]{article}
\usepackage{amssymb}

\newtheorem{theorem}{Theorem}

\newtheorem{definition}[theorem]{Definition}
\newtheorem{example}[theorem]{Example}

\newtheorem{lemma}[theorem]{Lemma}

\newtheorem{proposition}[theorem]{Proposition}

\newenvironment{proof}[1][Proof]{\noindent \textbf{#1.} }{\  \rule{0.5em}{0.5em}}
\input{tcilatex}
\begin{document}

\title{N-Legendre and N-Slant Curves in the Unit Tangent Bundle of Minkowski
Surfaces}
\author{\textsc{M.Bekar}, \textsc{F.Hathout and Y.Yayli}}
\date{}
\maketitle

\begin{abstract}
Let $(\mathbb{M}_{1}^{2},g)$ be a Minkowski surface and $\left( T_{1}\mathbb{%
M}_{1}^{2},g_{1}\right) $ its unit tangent bundle endowed with the
pseudo-Riemannian induced Sasaki metric. We extend in this paper the study
of the N-Legendre and N-slant curves which the inner product of normal
vector and Reeb vector is zero and nonzero constant respectively in $\left(
T_{1}\mathbb{M}_{1}^{2},g_{1}\right) ,$ given in \cite{hmy}, to the
Minkowski context and several important characterizations of these curves
are given.\newline
\end{abstract}

\section{Introduction}

As it is well known, a N-slant curve is a class of curves which the angle
between their tangent vector field and Reeb vector field $\xi $ is constant
of a given contact metric manifold $\left( M,g,\phi ,\xi ,\eta \right) $. A
particular case can be given by taking the angle as $\frac{\pi }{2}$ which
result to the class called N-legendre curves. The class of N-slant curves is
naturally an extension of slant curves studied by many authors in(\cite{cc},%
\cite{cil},\cite{ft},\cite{zl},$\cdots )$\textbf{.}

The unit tangent bundle $T_{1}M$ of some surfaces $M$ is a hypersurface of
tangent bundle $TM$ and $3$-dimensional almost contact metric structure, see
(\cite{b},\cite{s}).

In (\cite{hmy}), the authors studied the characterization of N-legendre and
N-slant curves in a non geodesic Euclidian case. How about Minkowski context?

Let $\mathbb{M}_{1}^{2}$ be a Minkowski surface endowed with
pseudo-Riemannian metric $g.$ The unit tangent bundle $T_{1}\mathbb{M}%
_{1}^{2}$ is a pseudo-Riemannian $3$-dimensional manifold endowed with an
almost contact metric structure denoted by $\left( T_{1}\mathbb{M}%
_{1}^{2},g_{1},\phi ,\xi ,\eta \right) .$

The purpose of the paper is to generalize the result of \cite{hmy} to
Minkowski space. In fact, in $\left( T_{1}\mathbb{M}_{1}^{2},g_{1},\phi ,\xi
,\eta \right) $ of Minkowski surface $\mathbb{M}_{1}^{2},$ we derive some
characterization theorems about $\tilde{N}$-legendre and $\tilde{N}$-slant
classes of curves in almost contact metric structure $\left( T_{1}\mathbb{M}%
_{1}^{2},g_{1},\phi ,\xi ,\eta \right) $ and some examples are given
especially in the de Sitter $\mathbb{S}_{1}^{2}$, anti de Sitter $\mathbb{H}%
_{1}^{2}$ and flat minkowski space cases.

\section{Preliminaries}

Let $\mathbb{M}_{1}^{2}$ be a three-dimensional Minkowski space endowed with
the standard metric given by 
\[
g=dx_{1}^{2}+dx_{2}^{2}-dx_{3}^{2}
\]%
where $(x_{1},x_{2},x_{3})$ is a standard rectangular coordinate system of $%
\mathbb{M}_{1}^{2}$. An arbitrary vector $x=(x_{1},x_{2},x_{3})$ in $\mathbb{%
M}_{1}^{2}$ can be one of the following three Lorentzian causal characters;
it can be spacelike if $g(x,x)>0$ or timelike if $g(x,x)<0$ or null
(lightlike) if $g(x,x)=0.$ A curve $\alpha :I\subset \mathbb{R}\rightarrow 
\mathbb{M}_{1}^{2}$ with arc-length parameter $s$ can be a locally
spacelike, timelike or null (lightlike), if all of its velocity vectors $%
\alpha ^{\prime }(s)$ are respectively spacelike, timelike or null
(lightlike). A spacelike or timelike curve $\alpha (s)$ is a non null unit
speed curve satisfying $g(\alpha ^{\prime }(s),\alpha ^{\prime }(s))=1$ and $%
-1$ respectively.

The Minkowski wedge product of two vectors in $\mathbb{M}_{1}^{2}$ is
defined by%
\[
x\times _{1}y=\left \vert 
\begin{array}{ccc}
\overrightarrow{i} & \overrightarrow{j} & -\overrightarrow{k} \\ 
x_{1} & x_{2} & x_{3} \\ 
y_{1} & y_{2} & y_{3}%
\end{array}%
\right \vert 
\]%
where $x=(x_{1},x_{2},x_{3})$ and $y=(y_{1},y_{2},y_{3})$ are any vectors in 
$M.$ The norm of $x$ is denoted by $\left \Vert x\right \Vert =\sqrt{g(x,x)}$
and two vectors $x$ and $y$ in $\mathbb{M}_{1}^{2}$ are said to be
orthogonal, if $g(x,y)=0.$ The de Sitter and anti de Sitter spaces are
respectively given by%
\begin{eqnarray*}
\mathbb{S}_{1}^{2}(\mathbf{r}) &=&\{x=(x_{1},x_{2},x_{3})\in \mathbb{M}%
_{1}^{2}\mid g(x,x)=\mathbf{r}^{2}\} \mathnormal{\ Lorentzian\ sphere} \\
\mathbb{H}_{0}^{2}(\mathbf{r}) &=&\{x=(x_{1},x_{2},x_{3})\in \mathbb{M}%
_{1}^{2}\mid g(x,x)=-\mathbf{r}^{2}\} \mathnormal{\ Hyperbolic\ sphere}.
\end{eqnarray*}%
Let $\alpha :I\subset \mathbb{R}\rightarrow \mathbb{M}_{1}^{2}$ be a non
null curve with arc-length parameter $s$. The moving Frenet-Serret formula
in $\mathbb{M}_{1}^{2}$ is given by%
\begin{equation}
\left( 
\begin{array}{c}
T^{\prime } \\ 
N^{\prime } \\ 
B^{\prime }%
\end{array}%
\right) =\left( 
\begin{array}{ccc}
0 & \varkappa  & 0 \\ 
-\varepsilon _{1}\varepsilon _{2}\varkappa  & 0 & \tau  \\ 
0 & -\varepsilon _{2}\varepsilon _{3}\tau  & 0%
\end{array}%
\right) \left( 
\begin{array}{c}
T \\ 
N \\ 
B%
\end{array}%
\right)   \label{1}
\end{equation}%
where $\tau (s)$ is the torsion of the curve $\alpha $ at $s$ and $%
g(T,T)=\varepsilon _{1}=\pm 1,$ $g(N,N)=\varepsilon _{2}=\pm 1$ and $%
g(B,B)=-\varepsilon _{1}\varepsilon _{2}.$ We write for this moving frame%
\[
T\times _{1}N=\varepsilon _{1}\varepsilon _{2}B,\ N\times _{1}B=-\varepsilon
_{1}T\mathnormal{\ and\ }B\times _{1}T=-\varepsilon _{2}N
\]%
If $\mathbb{M}_{1}^{1}$ is Minkowski surface endowed with the standard
metric given by 
\[
g=dx_{1}^{2}-dx_{2}^{2}
\]%
then the moving Frenet-Serret formula in $\mathbb{M}_{1}^{1}$ of\ the non
null curve $\alpha $ with arc-length parameter $s$ is%
\begin{equation}
\left( 
\begin{array}{c}
T^{\prime } \\ 
N^{\prime }%
\end{array}%
\right) =\left( 
\begin{array}{cc}
0 & \varepsilon _{2}\varkappa  \\ 
-\varepsilon _{1}\varkappa  & 0%
\end{array}%
\right) \left( 
\begin{array}{c}
T \\ 
N%
\end{array}%
\right)   \label{2}
\end{equation}%
and $g(T,T)=\varepsilon _{1}=\pm 1,$ $g(N,N)=\varepsilon _{2}=\pm 1$.

\subsection{Unit tangent bundle}

Let $(M,g)$ be a $n$-pseudo-Riemannian manifold, $\nabla $ the Levi-Civita
connection of $g$ and $TM$ its tangent bundle. The tangent space $T_{p}TM$
of $TM$ at $p=(x,u)$ splits into the horizontal and vertical subspaces $H_{p}
$ and $V_{p}$ as $T_{p}TM=H_{p}\oplus V_{p}$ with respect to $\nabla .$ The
horizontal (resp. vertical) lift of the vector field $X$ on $M$ to the point 
$p$ in $TM$ is the unique vector $X^{h}\in H_{p}$ (resp. $X^{v}\in V_{p})$
given by $\pi ^{\ast }(X^{h})=X$ (resp. $X^{v}(df)=X(f)$ for all smooth
functions $f$ on $M)$. The map $X\rightarrow X^{h}$ (resp. $X\rightarrow
X^{v})$ is an isomorphism between the vector spaces $T_{x}M$ and $H_{(x,u)}$
(resp. $V_{(p,u)})$. Each tangent vector $\tilde{X}\in T_{(p,u)}(TM)$ can be
written in the form $\tilde{X}=X^{h}+Y^{v}$. The unit tangent bundle $T_{1}M$
of a $n$ pseudo-Riemannian manifold $(M,g)$ is the $(2n-1)$-dimensional
hypersurface 
\[
T_{1}M=\{p=(x,u)\in TM\mid g_{x}(u,u)=\pm 1\}
\]%
where the canonical vertical vector field $\mathcal{U=}u^{v}$ is normal to $%
T_{1}M$. The tangential lift of a vector field $X$ on $M$ is a vector field $%
X^{t}$ tangent to $T_{1}M$ defined by%
\[
X^{t}=X^{v}-g(X,u)u.
\]%
Any vector $\tilde{X}\in T_{(x,u)}(T_{1}M)$ decompose in%
\begin{equation}
\tilde{X}=X^{h}+Y^{t}  \label{3}
\end{equation}%
where $X,Y\in \mathcal{X}(M)$ are uniquely determined vectors.

The induced pseudo-Riemannian Sasaki metric $g_{1}^{S}$\ on $T_{1}M$ is
uniquely determined by 
\begin{equation}
\left \{ 
\begin{array}{l}
i)\ g_{1}^{s}(X^{h},Y^{h})=g(X,Y) \\ 
ii)\ g_{1}^{s}(X^{h},Y^{t})=0 \\ 
ii)\ g_{1}^{s}(X^{t},Y^{t})=g(X,Y)-g(X,u)g(Y,u)%
\end{array}%
\right.   \label{3.1}
\end{equation}%
Its Levi-Civita connection $\nabla _{1}$ is given by%
\begin{equation}
\left \{ 
\begin{array}{l}
1.\; \nabla _{1}\ _{X^{h}}Y^{h}=(\nabla _{X}Y)^{h}-\tfrac{1}{2}(R(X,Y)u)^{t}
\\ 
2.\; \nabla _{1}\ _{X^{h}}Y^{t}=(\nabla _{X}Y)^{t}+\tfrac{1}{2}(R(u,Y)X)^{h}
\\ 
3.\; \nabla _{1}\ _{X^{t}}Y^{h}=\tfrac{1}{2}(R(u,X)Y)^{h} \\ 
4.\; \nabla _{1}\ _{X^{t}}Y^{t}=-g(Y,u)X^{t}%
\end{array}%
\right.   \label{3.5}
\end{equation}%
for all vector fields $X$ and $Y$ on $(M,g),$ see \cite{tt} and \cite{s}.

The unit tangent bundle $T_{1}M$ have a contact structure $(g_{1}^{s},\phi
^{\prime },\xi ^{\prime },\eta ^{\prime })$ and\textbf{\ }at any point $(x,u)
$ we have%
\begin{equation}
\begin{array}{c}
\xi ^{\prime }=u^{h},\  \  \  \  \  \eta ^{\prime }(\tilde{X})=g_{1}^{s}(\tilde{X}%
,\xi ^{\prime }) \\ 
\phi ^{\prime }(\tilde{X},\tilde{Y})=g_{1}^{s}(\tilde{X},\phi ^{\prime }%
\tilde{Y})=2d\eta ^{\prime }(\tilde{X},\tilde{Y}) \\ 
\phi ^{\prime }(X^{h})=X^{t}\text{,\  \  \  \  \ }\phi ^{\prime
}(X^{t})=-X^{h}+\eta ^{\prime }(X^{h})\xi ^{\prime }%
\end{array}%
\end{equation}%
where $X\in \mathcal{X}(M)$ and $\tilde{X},\tilde{Y}\in \mathcal{X}(TM),$
see \cite{b},\cite{y}.

Then, $T_{1}M$ is a contact metric pseudo-Riemannian manifold with a contact
metric structure denoted by $(g_{1},\phi ,\xi ,\eta )$ such that%
\begin{equation}
\eta =\frac{1}{2}\eta ^{\prime },\  \  \  \xi =\frac{1}{2}\xi ^{\prime },\  \  \
\phi =\phi ^{\prime }\text{,\  \  \ }g_{1}=\frac{1}{4}g_{1}^{s}  \label{4.5}
\end{equation}

\section{N-legendre and N-slant curves in $T_{1}\mathbb{M}_{1}^{2}$}

We suppose in this section that $(\mathbb{M}_{1}^{2},g)$ is a Minkowski
surface endowed with pseudo-Riemannian metric $g$, $\mathbb{\sigma }(t)$ its
sectional curvature, $\gamma :I\subset \mathbb{R}\rightarrow \mathbb{M}%
_{1}^{2}$ is a curve and $\tilde{\gamma}\left( t\right) =\left( \gamma
\left( t\right) ,X\left( t\right) \right) $ is a curve in 3-dimensional
contact metric pseudo-Riemannian manifold $\left( T_{1}\mathbb{M}%
_{1}^{2},g_{1},\phi ,\xi ,\eta \right) $ given in formula (\ref{4.5}). The
Frenet frame apparatus of $\tilde{\gamma}$ is $\left( \tilde{T},\tilde{N},%
\tilde{B},\tilde{\kappa},\tilde{\tau}\right) $ satisfies the formula (\ref{1}%
).

\bigskip

A general definition of the curves N-Legendre and N-slant in almost contact
metric structure manifold $\left( M,g,\phi ,\xi ,\eta \right) $ can be given
in the following definitions;

\begin{definition}[\protect \cite{cil},\protect \cite{zl}]
Let $\left( M,g,\phi ,\xi ,\eta \right) $ be an almost contact metric
structure. The curve $\gamma $ is a slant curve if the tangent vector field $%
T$ of $\gamma $ and $\xi $ satisfy%
\[
g(T,\xi )=c,\text{ nonzero constant} 
\]%
If $c=0,$ the curve $\gamma $ is called legendre curve.
\end{definition}

\begin{definition}[\protect \cite{hmy}]
Let $\gamma $ be a curve in an almost contact metric structure manifold $%
\left( M,g,\phi ,\xi ,\eta \right) .$ The curve is called a N-legendre
(resp. N-slant) curve if the normal vector field $N$ of $\gamma $ and $\xi $
satisfy%
\[
g(N,\xi )=0\  \ (\text{resp. }g(N,\xi )=c\  \text{nonzero constant}) 
\]
\end{definition}

\bigskip

Let $\tilde{\gamma}\left( t\right) $ be a non null (i.e., spacelike or
timelike) in the unit tangent bundle $\left( T_{1}\mathbb{M}%
_{1}^{2},g_{1},\phi ,\xi ,\eta \right) $ of Minkowski surface $\mathbb{M}%
_{1}^{2},$ parameterized by the arc length with Frenet frame apparatus $%
\left( \tilde{T},\tilde{N},\tilde{B},\tilde{\kappa},\tilde{\tau}\right) $,
then%
\begin{eqnarray}
\tilde{T}(t) &=&\frac{d\gamma ^{i}}{dt}\partial _{x^{i}}+\frac{dX^{i}}{dt}%
\partial _{u^{i}}  \label{6} \\
&=&(\frac{d\gamma ^{i}}{dt}(\partial _{x^{i}})^{h}+\frac{dX^{i}}{dt}+\frac{%
d\gamma ^{i}}{dt}X^{k}\Gamma _{jk}^{i})\partial _{u^{i}})(\gamma (t)) 
\nonumber \\
&=&(E^{h}+(\nabla _{E}X)^{v})(\gamma (t))  \nonumber
\end{eqnarray}%
where $E=\gamma ^{\prime }(t).$

Using the formulas (\ref{3.1}) and (\ref{4.5}), the Lorentzian angle $\theta 
$ between $\tilde{T}$ and $\xi =2X^{h}$ in $(T_{1}\mathbb{M}_{1}^{2},g_{1})$
is given by%
\begin{equation}
g_{1}(\tilde{T},\xi )=\tfrac{1}{2}g(E,X)=L(\theta )  \label{6.5}
\end{equation}%
where $L(\theta )$ is%
\begin{eqnarray}
&&\text{i.\ }\cos \theta \text{ if }\tilde{T}\text{ and }\xi \text{ are
spacelike vectors that span a spacelike vector subspace,}  \nonumber \\
&&\text{ii.\ }\cosh \theta \text{ if }\tilde{T}\text{ and }\xi \text{ are
spacelike vectors that span a timelike vector subspace,}  \nonumber \\
&&\text{iii.\ }\sinh \theta \text{ if }\tilde{T}\text{ and }\xi \text{ have
a different casual character.}  \label{7}
\end{eqnarray}

Differentiating the Equation (\ref{6.5}) with respect to $s$ and using the
formulas (\ref{3.1}),(\ref{4.5}) and (\ref{6}), we have%
\begin{eqnarray*}
\frac{d}{ds}g_{1}(\tilde{T},\xi ) &=&g_{1}(\nabla _{1}\ _{\tilde{T}(t)}%
\tilde{T},\xi )+g_{1}(\tilde{T},\nabla _{1}\ _{\tilde{T}(t)}\xi ) \\
&=&\tilde{\kappa}g_{1}(\tilde{N},\xi )+g_{1}(\tilde{T}(t),(2\nabla
_{E}X+R(u,\nabla _{E}X)u)^{h}+(R(E,u)u)^{t}) \\
&=&\tilde{\kappa}g_{1}(\tilde{N},\xi )+8g(E,\nabla _{E}X)-8R(E,X,X,\nabla
_{E}X) \\
&=&\theta ^{\prime }L^{\prime }(\theta )
\end{eqnarray*}%
and 
\begin{equation}
g_{1}(\tilde{N},\xi )=\frac{1}{\tilde{\kappa}}(8R(E,X,X,\nabla
_{E}X)-8g(E,\nabla _{E}X))-\frac{\theta ^{\prime }L^{\prime }(\theta )}{%
\tilde{\kappa}}  \label{7,5}
\end{equation}%
where $\xi (t)=2X^{h}$ and $R$ is curvature tensor of $\mathbb{M}_{1}^{2}.$

Let $(T,N)$ be the Frenet frame on $\gamma $ in $\mathbb{M}_{1}^{2}$ given
in (\ref{2}). The unit vector $X$ can be expressed from (\ref{6.5}), by%
\begin{equation}
X=\frac{2}{r}L(\theta )\ T+2\beta N  \label{8}
\end{equation}%
where $\beta $ is a $C^{\infty }$ function obtained from%
\[
\frac{4\varepsilon _{1}}{r^{2}}L^{2}(\theta )+4\varepsilon _{2}\beta
^{2}=\varepsilon _{X}
\]%
and%
\begin{equation}
\beta =\pm \frac{1}{r}\sqrt{\varepsilon _{X}\varepsilon _{2}(\tfrac{r}{2}%
)^{2}-\varepsilon _{1}\varepsilon _{2}L^{2}(\theta )}  \label{8,5}
\end{equation}%
where $\varepsilon _{X}=g(X,X)=\pm 1.$

The differentiation of the Equation (\ref{8}) with respect to $s$, gives%
\begin{eqnarray}
\nabla _{E}X &=&2(\frac{L(\theta )}{r})^{\prime }T+2\kappa L(\theta
)N+2\beta ^{\prime }N+2\varepsilon _{2}r\beta \kappa T  \label{9} \\
&=&2\left( (\frac{L(\theta )}{r})^{\prime }+\varepsilon _{2}r\beta \kappa
\right) T+2(\kappa L(\theta )+\beta ^{\prime })N  \nonumber
\end{eqnarray}%
From the orthogonality of the vectors $X$ and $\nabla _{E}X$ (i.e., $%
g(X,\nabla _{E}X)=0$) and using (\ref{6.5}) and (\ref{9}), the vector $E$ is
given by%
\begin{equation}
E=2L(\theta )\ X+\frac{2r}{\left \Vert \nabla _{E}X\right \Vert }\left( (\frac{%
L(\theta )}{r})^{\prime }+\varepsilon _{2}r\beta \kappa \right) \nabla _{E}X
\label{10}
\end{equation}%
It follows that 
\begin{eqnarray}
R(E,X,X,\nabla _{E}X) &=&2r\left( (\frac{L(\theta )}{r})^{\prime
}+\varepsilon _{2}r\beta \kappa \right) \frac{R(\nabla _{E}X,X,X,\nabla
_{E}X)}{\left \Vert \nabla _{E}X\right \Vert _{1}^{2}}  \label{10,5} \\
&=&2r\left( (\frac{L(\theta )}{r})^{\prime }+\varepsilon _{2}r\beta \kappa
\right) \sigma \mathbb{(}s\mathbb{)}  \nonumber
\end{eqnarray}%
where $\mathbb{\sigma }(s)$ is the sectional curvature of $\mathbb{M}%
_{1}^{2}.$

Substituting (\ref{8,5}),(\ref{9}),(\ref{10} and (\ref{10,5}) into (\ref{7,5}%
), we have%
\begin{equation}
g_{1}(\tilde{N},\xi )=16r\frac{(1-\mathbb{\sigma }(s))}{\tilde{\kappa}}%
\left( (\frac{L(\theta )}{r})^{\prime }\pm \varepsilon _{2}r\kappa \sqrt{%
\varepsilon _{X}\varepsilon _{2}(\tfrac{r}{2})^{2}-\varepsilon
_{1}\varepsilon _{2}L^{2}(\theta )}\right) -\frac{\theta ^{\prime }L^{\prime
}(\theta )}{\tilde{\kappa}}  \label{11}
\end{equation}

\section{Geometric consequences}

\begin{proposition}
Let $T_{1}\mathbb{S}_{1}^{2}$ be a unit tangent bundle of a the de Sitter
space $\mathbb{S}_{1}^{2}.$ All legendre and slant non geodesic curves with
any casual characterization are $\tilde{N}$-legendre curves.
\end{proposition}

\begin{proof}
Let $\tilde{\gamma}\left( t\right) =\left( \gamma \left( t\right) ,X\left(
t\right) \right) $ be a legendre or slant curve in $T_{1}\mathbb{S}_{1}^{2}$
parameterized by the arc length. In $\mathbb{S}_{1}^{2}$, the sectional
curvature is constant and $\mathbb{\sigma }(t)=1$, from (\ref{11}) we have 
\[
g_{1}(\tilde{N},\xi )=-\frac{\theta ^{\prime }L^{\prime }(\theta )}{\tilde{%
\kappa}}
\]%
If $\tilde{\gamma}$ is legendre or slant non geodesic curve then the angle $%
\theta $ is constant. Therefore, $g_{1}(\tilde{N},\xi )=0,$ i.e., $\tilde{%
\gamma}$ is $\tilde{N}$-legendre.
\end{proof}

\begin{theorem}
Let $\tilde{\gamma}\left( t\right) =\left( \gamma \left( t\right) ,X\left(
t\right) \right) $ be not a slant curve in $T_{1}\mathbb{S}_{1}^{2}$, then $%
\tilde{\gamma}$ is $\tilde{N}$-slant curve if the angle $\theta $ satisfies,%
\newline
i. $\theta =\arg \cosh c\tint \tilde{\kappa}\ ,$ if the vectors $\tilde{T}$
and $\xi $ are both space or time staying inside the same time-conic,\newline
ii. $\theta =\arccos c\tint \tilde{\kappa}\ ,$ if the vectors $\tilde{T}$
and $\xi $ are space staying inside the same space-conic,\newline
iii. $\theta =\arg \sinh c\tint \tilde{\kappa}\ ,$ if the vectors $\tilde{T}$
is space and $\xi $ is time,\newline
where $c$ is nonzero constant.
\end{theorem}

\begin{proof}
Let $\tilde{\gamma}\left( t\right) =\left( \gamma \left( t\right) ,X\left(
t\right) \right) $ be a non slant curve in $T_{1}\mathbb{S}_{1}^{2}$
parameterized by the arc length and $\sigma (t)=1$, from the definition (2)
and the formula (\ref{11}), we have%
\[
g_{1}(\tilde{N},\xi )=-\frac{L(\theta )^{\prime }}{\tilde{\kappa}}=\bar{c}\
\  \text{ constant}
\]%
If $\tilde{T}$ and $\xi $ are both space or time staying inside the same
time-conic, then $L(\theta )=\cosh \theta $ and 
\begin{eqnarray*}
(\cosh \theta )^{\prime } &=&-\bar{c}\tilde{\kappa} \\
\theta  &=&\arg \cosh c\int \tilde{\kappa}
\end{eqnarray*}%
Using the formula (\ref{7}) we have immediately ii and iii.
\end{proof}

\begin{proposition}
Let $\tilde{\gamma}=\left( \gamma ,X\right) $ be a slant curve in $T_{1}%
\mathbb{M}_{1}^{2}$ where $\mathbb{M}_{1}^{2}$ is not de Sitter space. The
curve $\tilde{\gamma}$ with any casual characterization is $\tilde{N}$-slant
if and only if the ratio%
\[
(1-\mathbb{\sigma })\frac{\kappa }{\tilde{\kappa}}\  \text{is nonzero constant%
} 
\]
\end{proposition}

\begin{proof}
The proof is a direct consequence from the definition (2) and the Equations (%
\ref{11})%
\begin{eqnarray*}
g_{1}(\tilde{N},\xi ) &=&16r\frac{(1-\mathbb{\sigma }(s))}{\tilde{\kappa}}%
\left( \pm \varepsilon _{2}r\kappa \sqrt{\varepsilon _{X}\varepsilon _{2}(%
\tfrac{r}{2})^{2}+L^{2}(\theta )}\right)  \\
&=&\pm 16r\left( \varepsilon _{2}r\sqrt{\varepsilon _{X}\varepsilon _{2}(%
\tfrac{r}{2})^{2}+L^{2}(\theta )}\right) \frac{(1-\mathbb{\sigma }(s))}{%
\tilde{\kappa}}\kappa  \\
&=&c\frac{(1-\mathbb{\sigma }(s))\kappa }{\tilde{\kappa}}
\end{eqnarray*}
\end{proof}

\begin{example}
In $\mathbb{H}_{1}^{2}$ and $\mathbb{S}_{1}^{2}(\mathbf{r})_{(\mathbf{r}\neq
1)},$ all slant curves are $\tilde{N}$-slant if and only if the ratio $\frac{%
\kappa }{\tilde{\kappa}}$ is nonzero constant.
\end{example}

\begin{theorem}
Let $\mathbb{M}_{1}^{2}$ be a time (resp. space) like surface. Let $\tilde{%
\gamma}\left( t\right) =\left( \gamma \left( t\right) ,X\left( t\right)
\right) $ be a curve in $T_{1}\mathbb{M}_{1}^{2}$ where $\gamma $ is a curve
of velocity $2$, $X$ a has different casual character with $\gamma $ (resp.
space like vector) and $\tilde{T}$,$\xi $ satisfies the formula (\ref{7} ii,
iii) (resp.(\ref{7} i)).\newline
1. The curve $\tilde{\gamma}$ is $\tilde{N}$-legendre curve if and only if%
\[
\theta =\int \kappa \bar{\sigma}(s)dt
\]%
\newline
2. The curve $\tilde{\gamma}$ is $\tilde{N}$-slant curve if and only if%
\[
L(\theta )^{\prime }\pm \kappa \bar{\sigma}(s)L(\theta )=c\tilde{\kappa}
\]%
where $\bar{\sigma}(s)=\frac{64(1-\mathbb{\sigma }(s))}{(15-16\mathbb{\sigma 
}(s))}$ and $L(\theta )$ is $\sinh \theta $ or $\cosh \theta $ (resp. $\cos
\theta ).$
\end{theorem}

\begin{proof}
Let $\mathbb{M}_{1}^{2}$ be a time (resp. space) like surface. If $\tilde{%
\gamma}\left( t\right) =\left( \gamma \left( t\right) ,X\left( t\right)
\right) $ is a curve in $T_{1}\mathbb{M}_{1}^{2}$, $\gamma $ with a velocity 
$2$ and $X$ has a different casual character with $\gamma $ (resp. space
like vector)$,$we have%
\[
\varepsilon _{x}=-\varepsilon _{1}=\varepsilon _{2}\text{ (resp. }%
\varepsilon _{x}=\varepsilon _{1}=\varepsilon _{2})
\]%
Using the Equation (\ref{11}), we get%
\begin{eqnarray*}
g_{1}(\tilde{N},\xi ) &=&16\frac{(1-\mathbb{\sigma }(s))}{\tilde{\kappa}}%
\left( \theta ^{\prime }L^{\prime }(\theta )\pm 4\kappa \sqrt{1+L^{2}(\theta
)}\right) -\frac{\theta ^{\prime }L^{\prime }(\theta )}{\tilde{\kappa}} \\
\text{(resp. }g_{1}(\tilde{N},\xi ) &=&16\frac{(1-\mathbb{\sigma }(s))}{%
\tilde{\kappa}}\left( \theta ^{\prime }L^{\prime }(\theta )\pm 4\kappa \sqrt{%
1-L^{2}(\theta )}\right) -\frac{\theta ^{\prime }L^{\prime }(\theta )}{%
\tilde{\kappa}})
\end{eqnarray*}%
If the vectors $\tilde{T}$, $\xi $ satisfies the formula (\ref{7} ii, iii)
(resp. (\ref{7} i)), then the function $L(\theta )$ is equal to $\sinh
\theta $ or $\cosh \theta $ (resp. $\cos \theta ).$ The $\tilde{N}$-legender
condition for $\tilde{\gamma}$ is%
\[
16\frac{(1-\mathbb{\sigma }(s))}{\tilde{\kappa}}\left( \theta ^{\prime }\pm
4\kappa \right) -\frac{\theta ^{\prime }}{\tilde{\kappa}}=0
\]%
and%
\[
\theta =64\int \tfrac{\kappa (1-\sigma (t))}{15-16\sigma (t)}dt
\]%
The $\tilde{N}$-slant condition for $\tilde{\gamma}$ is when the Lorentzian
angle $\theta $ satisfies the ODE%
\begin{eqnarray*}
g_{1}(\tilde{N},\xi ) &=&c\text{ and} \\
L(\theta )^{\prime }\pm \kappa \frac{64(1-\mathbb{\sigma }(s))}{(15-16%
\mathbb{\sigma }(s))}L(\theta ) &=&c\tilde{\kappa}
\end{eqnarray*}
\end{proof}

\bigskip

Now, we suppose in the sequel that the angle $\theta $ is linear (i.e., $%
\theta =at+b$), thus we have the following theorems;

\bigskip

\begin{lemma}
The reeb vector $\xi $ and the vector $X$ have the same casual characters
(i.e., space or time vector).
\end{lemma}

\begin{proof}
From the formulas (\ref{3.1}) and (\ref{4.5}),\textbf{\ }we have $g_{1}(\xi
,\xi )=g_{1}(2X^{h},2X^{h})=g(X,X)=\varepsilon _{X}.$
\end{proof}

\begin{theorem}
Let $\mathbb{M}_{1}^{2}$ be a space like surface and let $\tilde{\gamma}%
\left( t\right) =\left( \gamma \left( t\right) ,X\left( t\right) \right) $
be a not slant space curve in $T_{1}\mathbb{M}_{1}^{2}$ where $\gamma $ is a
curve of velocity $2$ and $\tilde{T},$ $\xi $ are spacelike vectors that
span a spacelike vector subspace. Then, the curve $\tilde{\gamma}$ is $%
\tilde{N}$-legendre if and only if%
\[
(1-\mathbb{\sigma }(s))\left( a\pm 4\kappa \right) =a/16
\]
\end{theorem}

\begin{proof}
Let $\mathbb{M}_{1}^{2}$ be a space like surface. If $\tilde{\gamma}\left(
t\right) =\left( \gamma \left( t\right) ,X\left( t\right) \right) $ is not a
slant curve in $T_{1}\mathbb{M}_{1}^{2}$, $\gamma $ with a velocity $2$ and $%
X$ is space like vector (i.e., $\varepsilon _{X}=1).$ From the Lemma (7) and 
$\mathbb{\sigma }(s)\neq 1,$ the Equation (\ref{11}) turns into%
\begin{eqnarray*}
16\frac{(1-\mathbb{\sigma }(s))}{\tilde{\kappa}}\left( aL^{\prime }(\theta
)\pm \varepsilon _{2}4\kappa \sqrt{\varepsilon _{X}\varepsilon
_{2}-\varepsilon _{1}\varepsilon _{2}L^{2}(\theta )}\right) -\frac{%
aL^{\prime }(\theta )}{\tilde{\kappa}} &=&0 \\
16(1-\mathbb{\sigma }(s))\left( L^{\prime }(\theta )\pm 4\kappa \sqrt{%
1-L^{2}(\theta )}\right) -aL^{\prime }(\theta ) &=&0
\end{eqnarray*}%
Taking account that $\tilde{T}$ and $\xi $ are space staying inside the same
space-conic%
\[
L(\theta )=\cos \theta ;\  \ L^{\prime }(\theta )=-\sin \theta ,
\]%
hence%
\begin{eqnarray*}
16(1-\mathbb{\sigma }(s))\left( aL^{\prime }(\theta )\pm 4\kappa \sqrt{%
1-L^{2}(\theta )}\right) -aL^{\prime }(\theta ) &=&0 \\
16(1-\mathbb{\sigma }(s))\left( -a\sin \theta \pm 4\kappa \sin \theta
\right) +a\sin \theta  &=&0 \\
(1-\mathbb{\sigma }(s))\left( a\pm 4\kappa \right)  &=&a/16.
\end{eqnarray*}
\end{proof}

\bigskip

Similar, the following theorem can be given for the time like case of $%
\mathbb{M}_{1}^{2}$.

\begin{theorem}
Let $\mathbb{M}_{1}^{2}$ be not a de Sitter sphere time like surface and let 
$\tilde{\gamma}\left( t\right) =\left( \gamma \left( t\right) ,X\left(
t\right) \right) $ be not a slant a curve in $T_{1}\mathbb{M}_{1}^{2}$ where 
$\gamma $ is a curve of velocity $2$, $X$ has a different casual character
with $\gamma $ and $\tilde{T}$, $\xi $ satisfies the formula (\ref{7},i or
ii). Then the curve $\tilde{\gamma}$ is $\tilde{N}$-legendre if and only if%
\[
(1-\mathbb{\sigma }(s))\left( a\pm 4\kappa \right) =a/16
\]
\end{theorem}

\begin{proof}
Using the formula (\ref{7}) and taking account that $\varepsilon
_{X}=-\varepsilon _{1}=\varepsilon _{2},$ the proof is similar to the
Theorem 8.
\end{proof}

\begin{example}
Under the condition of the Theorem 8, if $\mathbb{M}_{1}^{2}=\mathbb{H}%
_{1}^{2}$ the $\tilde{N}$-legendre of $\tilde{\gamma}$ condition is that its
projection $\gamma $ has a constant curvature $\kappa =\pm \frac{15}{64}a.$
\end{example}

\bigskip

\begin{theorem}
Let $\mathbb{M}_{1}^{2}$ be a space like surface. Let $\tilde{\gamma}\left(
t\right) =\left( \gamma \left( t\right) ,X\left( t\right) \right) $ be not
slant a space curve in $T_{1}\mathbb{M}_{1}^{2}$ where $\gamma $ is a curve
of velocity $2$ and $X$ is space like vector (resp $X$ has a different
casual characters with the curves $\tilde{\gamma}$ and $\gamma $). Then the
curve $\tilde{\gamma}$ is $\tilde{N}$-slant if and only if%
\[
\theta =\arcsin \frac{c\tilde{\kappa}}{(a+16(\sigma (s)-1)(a\pm 4\kappa ))}%
;\ c\text{ is nonzero constant.}
\]%
when $\tilde{T}$ and $\xi $ span a spacelike vector subspace.
\end{theorem}

\begin{proof}
Let $\mathbb{M}_{1}^{2}$ be a space like surface. Let $\tilde{\gamma}\left(
t\right) =\left( \gamma \left( t\right) ,X\left( t\right) \right) $ be a not
slant space curve in $T_{1}\mathbb{M}_{1}^{2}$ where $\gamma $ is a curve of
velocity $2$ and $X$ is space like vector. If $\tilde{\gamma}$ is $\tilde{N}$%
-slant, the Equation (\ref{11}) becomes 
\begin{eqnarray*}
g_{1}(\tilde{N},\xi ) &=&c;\  \varepsilon _{X}=\varepsilon _{1}=\varepsilon
_{2}=1\text{ and} \\
c &=&16\frac{(1-\mathbb{\sigma }(s))}{\tilde{\kappa}}\left( L(\theta
)^{\prime }\pm \varepsilon _{2}4\kappa \sqrt{\varepsilon _{X}\varepsilon
_{2}-\varepsilon _{1}\varepsilon _{2}L^{2}(\theta )}\right) -\frac{%
aL^{\prime }(\theta )}{\tilde{\kappa}}
\end{eqnarray*}%
Using the Lemma (8), the vectors $\tilde{T}$ and $\xi $ span a spacelike
vector subspace, then%
\[
L(\theta )=\cos \theta ;\ L^{\prime }(\theta )=-\sin \theta 
\]%
and%
\[
c=16\frac{(1-\mathbb{\sigma }(s))}{\tilde{\kappa}}\left( -a\sin \theta \pm
4\kappa \sin \theta \right) +\frac{a\sin \theta }{\tilde{\kappa}}
\]%
hence%
\begin{eqnarray*}
c\tilde{\kappa} &=&16(1-\mathbb{\sigma }(s))\left( -a\sin \theta \pm 4\kappa
\sin \theta \right) +a\sin \theta  \\
&=&\sin \theta (a+16(\sigma (s)-1)(a\pm 4\kappa )) \\
\theta  &=&\arcsin \frac{c\tilde{\kappa}}{(a+16(\sigma (s)-1)(a\pm 4\kappa ))%
}
\end{eqnarray*}
\end{proof}

\begin{theorem}
Let $\mathbb{M}_{1}^{2}$ be a not de Sitter time like surface. Let $\tilde{%
\gamma}\left( t\right) =\left( \gamma \left( t\right) ,X\left( t\right)
\right) $ be a not slant space curve in $T_{1}\mathbb{M}_{1}^{2}$ where $%
\gamma $ is a curve of velocity $2$ and $X$ has a different casual
characters with the curves $\tilde{\gamma}$ and $\gamma $. Then the curve $%
\tilde{\gamma}$ is $\tilde{N}$-slant if and only if\newline
i.%
\[
\theta =\arg \sinh \frac{c\tilde{\kappa}}{(a-16(\mathbb{\sigma }(s)-1)\left(
a\pm 4\kappa \right) )};\ c\text{ is nonzero constant.}
\]%
when $\tilde{T}$ and $\xi $ satisfies the formula (\ref{7} ii).\newline
ii. 
\[
\theta =\arg \cosh \frac{c\tilde{\kappa}}{(a-16(\mathbb{\sigma }(s)-1)\left(
a\pm 4\kappa \right) )};\ c\text{ is nonzero constant.}
\]%
when $\tilde{T}$ and $\xi $ satisfies the formula (\ref{7} iii).
\end{theorem}

\begin{proof}
Using the formulas (\ref{7}) and (\ref{11}), the Lemma (8) and the casual
characters%
\[
\varepsilon _{X}=-\varepsilon _{1}=\varepsilon _{2}
\]%
the proof is similar to the proof of the Theorem (12).
\end{proof}

\bigskip

\textsc{Murat BEKAR}{\small , Department of mathematics and computer
sciences,}

{\small Necmettin Erbakan university, 42090 Konya, Turkey. Email:
mbekar@konya.edu.tr}

\textsc{Fouzi HATHOUT}{\small , Department of mathematics, Sa\"{\i}da
University, 2000}

{\small Saida, Algeria. Email: f.hathout@gmail.com}

\textsc{Yusuf YAYLI}{\small , Department of mathematics, Ankara university,
06100}

{\small Ankara, Turkey. Email: yayli@science.ankara.edu.tr}

\end{document}